\documentclass{article}

\usepackage{epigraph}
\usepackage[utf8]{inputenc}
\usepackage[english]{babel}
\usepackage{amsfonts}
\usepackage{amssymb}
\usepackage{amsmath}
\usepackage{xypic}
\usepackage[all]{xy}
\usepackage{mathrsfs}
\usepackage{latexsym}
\usepackage{amscd}

\usepackage[mathscr]{eucal}

\usepackage{amsthm}

\usepackage{setspace}                    

\usepackage{tocbibind}

\usepackage{hyperref}

\newtheorem{prop}{Proposition}[section]

\newtheorem{thm}[prop]{Theorem}
\newtheorem{theorem}[prop]{Theorem}

\newtheorem{corollary}[prop]{Corollary}

\theoremstyle{definition}

\newtheorem{definition}[prop]{Definition}

\theoremstyle{remark}
\newtheorem{rmk}[prop]{Remark}
\newtheorem{remark}[prop]{Remark}

\newtheorem{example}[prop]{Example}

\newtheorem{question}[prop]{Question}

\newcommand{\G}{\mathbb{G}}

\newcommand{\Sym}{\mathrm{Sym}}

\usepackage{tikz}
\usetikzlibrary{arrows}

\newcommand{\pp}{\mathbb{P}}

\newcommand{\oo}{\mathcal{O}}

\newcommand{\rk}{\mathrm{rk}}

\newcommand{\CC}{\mathbb{C}}

\def\To{\longrightarrow}

\def\rmapdown#1{\Big\downarrow\rlap{$\vcenter{\hbox{$\scriptstyle
#1$}}$}}
\def\lmapdown#1{\Big\downarrow\llap{$\vcenter{\hbox{$\scriptstyle
#1\;\;\,$}}$}}

\def\lmapup#1{\Big\uparrow\llap{$\vcenter{\hbox{$\scriptstyle
#1\;\;\,$}}$}}

\def\squaremap#1#2#3#4#5#6#7#8{\mathop{
\begin{array}{ccccc}
#1 & \stackrel{#2}{\To}&  #3\\
\lmapdown{#4} & { } & \rmapdown{#5}\\
#6 & \stackrel{#7}{\To} & #8
\end{array}
}\nolimits}

\title{Holomorphic symmetric differentials and a birational characterization of Abelian Varieties}
\author{Ernesto C. Mistretta}
\date{}

\begin{document}

\maketitle

%

\begin{abstract}
A generically generated  vector bundle on a smooth projective variety
yields a rational map to a Grassmannian,
called Kodaira map.
We answer a previous question, raised by the asymptotic behaviour of such maps,
giving rise to a  birational characterization of abelian varieties.

In particular we prove that, under the conjectures of the Minimal Model Program,
a smooth projective variety is birational to an abelian variety if and only if it has Kodaira dimension 0 and 
some symmetric power of its cotangent sheaf is generically generated by its global sections.

\end{abstract}

\section{Introduction}

The aim of this work is to answer positively a question raised in the framework of the investigation on stable base loci for vector bundles started in \cite{bkkmsu}.

In the recent work 
\cite{mistrurbi} we extended the construction of the Iitaka fibration to the case of higher rank vector bundles,
and the natural setting to do so is to consider asymptotically generically generated (AGG) vector bundles. 
A vector bundle $E$ on a projective variety $X$ 
is said to be \emph{asymptotically generically generated} when 
some symmetric power $\Sym^m E$ 
is generated over a nonempty open subset $U \subseteq X$ by its global sections $H^0(X, \Sym^m E)$.

In the same work we proved that if the cotangent bundle $\Omega_X$ of a smooth projective variety $X$ with Kodaira dimension $0$ is strongly semiample,
then the variety must be isomorphic to an abelian Variety (we say that a vector bundle $E$ on a projective variety $X$ 
is \emph{strongly semiample} when 
some symmetric power $\Sym^m E$ 
is globally generated). That led to consider the question whether the generic condition AGG can give a birational characterisation of  abelian varieties.

In the present work we give a positive answer to this question under the hypothesis that the main conjecture of the minimal 
model program be satisfied:

\begin{theorem}
\label{mainthnm}
Let $X$ be a smooth variety of Kodaira dimension $0$. Suppose that $X$ admits a minimal model $Y$ and that abundance conjecture holds for $Y$. If the vector bundle $\Omega_X$ is Asymptotically Generically Generated then $X$ is birational to an abelian variety.
\end{theorem}

In particular in dimension at most 3 we obtain:

\begin{corollary}
\label{maincor}

Let $X$ be a smooth projective variety of dimension $n \leqslant 3$. The following conditions are equivalent:

\begin{enumerate}
\item The variety $X$ is birational to an abelian variety, in particular the Albanese map $a_X: X \to \mathrm{Alb}(X)$ is surjective and birational.
\item The variety $X$ has Kodaira dimension $\mathrm{kod}(X, K_X)=0$ and the cotangent bundle $\Omega_X$ is asymptotically
generically  generated.
\item The cotangent bundle $\Omega_X$ has stable base locus 
$\mathbb{B}(\Omega_X) \neq X$ 
and Iitaka  dimension $\mathrm{k}(X, \Omega_X)=0$.

\end{enumerate}

\end{corollary}

The difference between the second and third statemnts in this corollary lies in the fact that,
according to our constructions in \cite{mistrurbi}, in the strongly semiample case the dimension $k(X, E)$ 
of the image of a Kodaira map 
\[
\varphi_m \colon X \to \G r (H^0 (X, \Sym^m E), \sigma_m(r))
\]
for a globally generated rank $\sigma_M(r)$  vector bundle $\Sym^m E$,   
is always the same as the Iitaka dimension $k(X, \det E)$ of the determinant of $E$, while we cannot say whether this is true for an 
asymptotically generically generated vector bundle as the cotangent bundle in the hypotheses of the Corollary.

In order to prove the  theorem we first prove that within the hypotheses we can apply a criterion due to Greb-Kebekus-Peternell (cf. \cite{GKP}) to show that the minimal model $Y$ of $X$ is a quotient of an abelian variety,
then we show that on such a quotient the cotangent bundle of a resolution cannot be Asymptotically Generically Generated unless the quotient is an abelian variety itself.

\subsection{Acknowledgments} A special thank Gian Pietro Pirola, Antonio Rapagnetta and Francesco Polizzi  for extremely helpful and very pleasant conversations that lead to the present work.
And  to Andreas H\"oring, 
whose comments to a previous work raised the question answered here.

\section{Notation and previous results}

We will work with projective varieties 
over the field $\CC$ of complex numbers. 
When the variety $X$ is smooth,
we will denote $\Omega_X$ the cotangent bundle,
and identify it with the sheaf of K\"ahler differentials on $X$.
As usual we will denote $\Omega^p_X = \bigwedge^p \Omega_X$ the higher exterior powers, identified with the sheaf of holomorphic $p$-forms.

We give in this section the main definitions and results for base loci and Kodaira maps for vector bundles,
most of the definitions can be found in \cite{mistrurbi}.

\subsection{Stable base locus}
\begin{definition}
Let $E$ be a vector bundle on a projective variety $X$.
\begin{enumerate}

\item We call \emph{base locus} of  $E$  the closed subset
\[
\mathrm{Bs}(E) := \{x \in X ~|~ ev_x \colon H^0(X, E) \to E(x) \textrm{ is not surjective} \} \subseteq X
\]
and \emph{stable base locus} the closed subset
\[
\mathbb{B} (E) := \bigcap_{m>0} \mathrm{Bs} (\Sym^m E) \subseteq X
\]

\item We say that a vector bundle $E$ on $X$ is \emph{strongly semiample} if $\mathbb{B}(E) = \emptyset$,
\emph{i.e.} if some symmetric power of $E$ is globally generated.

\item We say that a vector bundle $E$ on $X$ is \emph{generically generated} if ${Bs}(E) \neq X$,
\emph{i.e.} if $E$ is  generated over a nonempty open subset $U \subseteq X$ by its global sections $H^0(X, E)$.

\item We say that a vector bundle $E$ on $X$ is \emph{asymptotically generically generated} if $\mathbb{B}(E) \neq X$.
\end{enumerate}
\end{definition}

\begin{remark}

In general the definition of strong semiampleness is not equivalent to the usual definition of semiampleness:
it is stronger and not equivalent to the fact that $\oo_{\pp(E)}(1)$ is semiample. A simple counterexample to the equivalence  can be found in \cite {mistrurbi}.

\end{remark}

The following theorems are proved in \cite{mistrurbi}:

\begin{theorem}

Let $X$ be a smooth projective variety. Then $X$ is isomorphic to an abelian variety if and only if
the cotangent bundle is strongly semiample and $X$ has Kodaira dimension $0$.

\end{theorem}

\begin{theorem}

Let $X$ be a smooth projective surface. Then $X$ is birational to an abelian variety if and only if 
the cotangent bundle of $X$ is asymptotically generically generated and $X$ has Kodaira dimension $0$.

\end{theorem}

Those theorems lead to ask whether the following holds:

\begin{question}
Let $X$ be a smooth projective variety of Kodaira dimension $0$. Suppose that the cotangent bundle 
$\Omega_X$ is asymptotically globally generated, is $X$ birational to an abelian variety?

\end{question}

The main result in this work is an affirmative answer to this question, 
at least in the case where Minimal Model Program and Abundance Conjecture can be applied.

In order to prove these results, we apply our previous results in order to show that 
the theorem holds in the case a minimal model of $X$ is smooth,
then we show that within the hypotheses a singular minimal model cannot occur.

\section{Quotients of Abelian Varieties}

In this section we show that a smooth projective variety of Kodaira dimension $0$ with asymptotically generically generated cotangent bundle is birational to a quotient of an abelian variety.


We will use the following theorem to show that a minimal model of the variety we are dealing with is a quotient of an abelian variety:

\begin{theorem}[Greb-Kebekus-Peternell, \cite{GKP}]
\label{gkpthm}
Let $Y$ be a normal, complex, projective variety of dimension $n$, 
with at worst KLT
singularities. 
Assume that $Y$ is smooth in codimension 2, 
and assume that the canonical
divisor of $Y$ is numerically trivial, $K_Y \equiv 0$. 
Further, assume that there exist ample
divisors $H_1, \dots , H_{n-2}$ on $Y$ and a desingularization
$\pi \colon \widetilde{Y} \to Y$ such that 
$c_2(\Omega_{\widetilde{Y}}).H_1 \dots H_{n-2} = 0$. 
Then, there exist an abelian variety $A$ and a finite, surjective,
Galois morphism $A \to Y$ that is \'etale in codimension 2.

\end{theorem}

\begin{rmk}
We remark that by Kawamata \cite{kawamata} the condition that $K_Y \equiv 0$ is equivalent to the fact 
that $K_Y$ is torsion (therefore $K_Y$ is semiample and $Y$ has Kodaira dimension $0$).
Abundance conjecture predicts that if $K_Y$ is nef then it is semiample,
which for Kodaira dimension $0$ varieties $Y$ is the same as the equivalence between ``$K_Y$ nef'' and ``$K_Y \equiv 0$''.
\end{rmk}

\begin{remark}

For a smooth variety, the property of having an asymptotically generically generated cotangent bundle is a birationally invariant property,
contrary to having a strongly semiample cotangent bundle.

In fact, if the cotangent bundle $\Omega_X$ on a variety $X$ is globally gnerated, then the cotangent 
bundle on the blow up $\widetilde{X}$ of $X$ in a point is generically generated, but global sections do not generate it on the exceptional divisor,
and the same holds for its symmetric powers.

However, as the space of global sections $H^0(X,\Sym^k \Omega_X)$ of a symmetric power (or any tensor power) of the cotangent bundle 
is invariant by birational morphisms of smooth varieties,  it is immediate to check 
that also asymptotic generic generation of the cotangent bundle is a birational invariant of smooth projective varieties.

\end{remark}

Let us use the  result above in order to prove that the variety $X$ is birational to a quotient of an Abelian variety.

\begin{prop}

Let $X$ be a smooth projective variety with $kod(X)=0$. 
Suppose that the cotangent bundle $\Omega_X$ is asymptotically generically generated, 
and that $X$ admits a minimal model $Y$ that satisfies abundance conjecture 
(\emph{i.e.} $Y$ has terminal singularities and  $K_{Y}$ is numerically trivial).
Then, there exists an abelian variety $A$ and a finite, surjective,
Galois morphism $A \to Y$ that is \'etale in codimension 2.
In particular $X$ is birational to a quotient of an abelian variety by a finite group,
and the quotient map is \`etale in codimension $2$.

\end{prop}

\begin{proof}

Consider a minimal model $Y \sim_{bir} X$ such that 
$K_Y$ is numerically trivial. 
Choose a resolution of singularities 
$\pi \colon \widetilde{Y} \to Y$,
let us show that, for $H_1, \dots, H_{2n-2}$ ample divisors on $Y$, 
we have 
$c_2(\Omega_{\widetilde{Y}}).H_1 \dots H_{n-2} = 0$. 
As $Y$ has terminal singularities, $Y^{sing}$ has codimension at least $3$ in $Y$,
so we can suppose that 
$H_1 \cap \dots \cap H_{n-2} \cap Y^{sing} = \emptyset$.
As $\widetilde{Y}$ is smooth and birational to $X$ then 
$\Sym^k \Omega_{\widetilde{Y}}$ is generically generated for some $k>0$.

Now let us consider the base locus of $\Sym^k \Omega_{\widetilde{Y}}$:
this is a divisor in $\widetilde{Y}$, linearly equivalent to a multiple of the canonical divisor $K_{\widetilde{Y}}$. As the Kodaira dimension of 
$\widetilde{Y}$ is $0$ and $Y$ is terminal, then any multiple of the canonical divisor does not move and its support is contained in the exceptional locus of $\pi$. 
So
$\mathrm{Bs}(\Sym^k \Omega_{\widetilde{Y}}) \subset Exc(\pi)$.
In particular 
$(\Sym^k \Omega_{\widetilde{Y}})_{|H_1 \cap \dots \cap H_{n-2}}$
is generated by its global sections, its determinant is trivial, and therefore it is trivial.
Then $c_1(\Sym^k \Omega_{\widetilde{Y}}). H_1 \dots  H_{n-2} =0$
and $c_2(\Sym^k \Omega_{\widetilde{Y}}).H_1 \dots  H_{n-2} =0$.

\begin{sloppypar}
As $c_1(\Sym^k \Omega_{\widetilde{Y}})$ is a multiple of $c_1(\Omega_{\widetilde{Y}})$, and 
$c_2(\Sym^k \Omega_{\widetilde{Y}})$ is a combination of 
$c_1(\Omega_{\widetilde{Y}})^2$ and $c_2(\Omega_{\widetilde{Y}})$,
this implies that $c_1(\Omega_{\widetilde{Y}}).H_1 \dots H_{n-2} =0$
and 
\hbox{$c_2(\Omega_{\widetilde{Y}}).H_1  \dots  H_{n-2} =0$,}
and we can apply  Proposition \ref{gkpthm}.
\end{sloppypar}

\end{proof}

So we have shown that, for a Kodaira dimension $0$ variety $X$
(satisfying the MMP conjectures), 
if the cotangent is asymptotically generically generated
then a model of $X$ is a quotient of an abelian variety.
Let us distinguish two cases according to this quotient being 
smooth or singular.

\section{The smooth case}

In this section we show that a smooth projective variety $X$ with Kodaira dimension $0$ and asymptotically generically generated cotangent bundle is birational to an abelian variety if a minimal model of $X$ is smooth.

The results of this section make use of very similar constructions 
to some of the ones in \cite{mistrurbi}, 
we will recall anyway the constructions needed.

We have proven above that a smooth projective variety $X$ with Kodaira dimension $0$ (satisfying the MMP conjectures) and asymptotically generically generated cotangent bundle $\Omega_X$
has a minimal model $Y$ which is a quotient of an abelian variety.

Let us show that if the minimal model $Y$ is smooth 
then $Y$ is an abelian variety itself:

\begin{thm}
\label{smooth}

Let $X$ be a smooth projective variety, suppose that the cotangent bundle $\Omega_X$ is asymptotically generated by global sections, and that $X$ is 
birational to a smooth quotient $Y = A/G$ of an abelian variety by a finite group. 
Then $Y$ is an abelian variety.
	
\end{thm}

\begin{proof}

In order to prove that $Y = A/ G$ is an abelian variety we will prove that $G$ acts on $A$ by translations,
\emph{i.e.} we will prove that the action of $G$ on $H^0(A, \Omega_A)$ is trivial.

Since the quotient $f \colon A \to A/G = Y$ is smooth then $f$ is an étale map.
Therefore $f^* \Omega_Y \cong \Omega_A$,
and  $H^0(Y, \Sym^m \Omega_Y) \cong H^0(A, \Sym^m \Omega_A)^G$ 
for all $m>0$.
As $Y$ is a smooth variety birational to $X$,
then $\Omega_Y$ is asymptotically generically generated,
therefore $\dim H^0(Y, \Sym^k \Omega_Y) \geqslant \rk (\Sym^k \Omega_Y)$,  for some $k>0$.
So the inclusion
\[
H^0(Y, \Sym^k \Omega_Y) \cong H^0(A, \Sym^k \Omega_A)^G
\subseteq H^0(A, \Sym^k \Omega_A)
\]
must be an equivalence, as 
$\dim H^0(A, \Sym^k \Omega_A) = \rk \Sym^k \Omega_A$ for all $k>0$.

So the action 
of $G$ on $H^0(A, \Sym^k \Omega_A)$ is trivial,
and this can happen if and only if $G$ acts on $H^0(A, \Omega_A)$
through homothethies 
(\emph{i.e.} multiplication by roots of $1$):
in fact as $G$ is a finite group, then the action of an element $g \in G$
on $H^0(A, \Omega)$
is diagonalizable.
Given two eigenvectors $v,w \in H^0(A, \Omega_A)$, with eigenvalues $\lambda$ and $\mu$,
choosing $k>0$ such that $\lambda^k = \mu^k = 1$ and that $G$ acts trivially on 
$H^0(A, \Sym^k \Omega_A) = \Sym^k H^0(A, \Omega_A)$, we have
\[
g\cdot (v^{k-1}.w) = \lambda^{k-1}v^{k-1}.\mu w = 
(\frac{\mu}{\lambda})v^{k-1}.w = v^{k-1}.w
\]
so $\lambda = \mu$.

But in this case  the quotient will have singular points unless 
the action of $G$ is trivial on $H^0(A, \Omega_A)$:
if the action of an element $g\in G$ on 
$H^0(X, \Omega_A)$ is given by multiplication by $\lambda_g \neq 1$,
then $g$ acts on $A$ by $x  \mapsto \lambda_g .x + \tau_g$ for some $\tau_g \in A$,
therefore there is a point $y = (1-\lambda_g)^{-1}\tau_g \in A$ fixed by $g$.

Hence
$G$ acts by translations on the variety $A$, so the quotient $Y$
is an abelian variety.

\end{proof}

\begin{remark}

The proof of the last theorem is very similar to the proof that 
a Kodaira dimension 0 variety with strongly semiample cotangent bundle is isomorphic to an abelian variety given in \cite{mistrurbi},
we just observe that strong semiampleness is not needed here, and that the characterisation of abelian varieties appearing in 
\cite{mistrurbi} is a corollary of Theorem \ref{smooth}.
\end{remark}

\section{The singular case}

In this section we show that the cotangent bundle of
a resolution of singularities of 
a singular quotient $Y=A/G$ cannot be asymptotically generically generated.

We will use the following result on extensions of symmetric differentials:

\begin{theorem}[Greb-Kebekus-Kovacs, \cite{GKK} Corollary 3.2]
\label{gkkthm}
Let Y be a normal variety. Fix an integer $k >0$. 
Suppose that there exists a normal variety $W$ and a finite surjective morphism 
$\gamma \colon W \to Y$,
such that for any resolution of singularities $p \colon \widetilde{W} \to W$
with simple normal crossing  (snc) exceptional locus $F = Exc(p)$
the sheaf
\[p_* \Sym^k \Omega^1_{\widetilde{W}}(log F)
\]
is reflexive.
Then 
for any resolution of singularities $\pi \colon   \widetilde{Y} \to Y$
with snc exceptional locus $E = Exc(\pi)$
the sheaf
$\pi_* \Sym^k \Omega^1_{\widetilde{Y}}(log E)$ is reflexive.

\end{theorem}

\begin{rmk}

In \cite{GKK} the property  of the  sheaf $\pi_* \Sym^k \Omega^1_{\widetilde{Y}}(log E)$ being reflexive 
for any such resolution of singularities of $Y$
is stated by saying that  ``Extension theorem holds for $\Sym^k$-forms on $Y$'', and is considered in more generality 
for \emph{ reflexive tensor operations}
on differential 1-forms on logarithmic pairs $(Y, \Delta)$.

\end{rmk}

\begin{theorem}
\label{sing}

Let $X$ be  a smooth projective variety, and suppose that $X$ is birational to a quotient of an abelian variety 
$A$ by a finite group $G$, such that the quotient map is \'etale in codimension 1.
If $Y = A/G$ is singular, then the cotangent bundle $\Omega_X$ is not asymptotically generically generated.

\end{theorem}

\begin{proof}

We apply theorem \ref{gkkthm} to $Y$ and the finite map $$\gamma \colon A \to Y = A/G~.$$
As $A$ is smooth, the hypothesis trivially apply. Let us call $\pi \colon \widetilde{Y} \to Y$
a resolution of singularities of $Y$, with snc exceptional divisor $E$, 
then for any $m>0$
the sheaf
$\pi_* \Sym^m \Omega^1_{\widetilde{Y}}(log E)$ is reflexive. Let 
$F \subset Y$ be the locus where the map $\gamma$ is not \'etale,
this locus having codimension at least $2$.
So 
\[
H^0(Y, \pi_* \Sym^m \Omega^1_{\widetilde{Y}}(log E)) = 
H^0(Y \setminus (Y^{sing} \cup F), \pi_* \Sym^m \Omega^1_{\widetilde{Y}}(log E))
\]
and we have the following inclusions for all $m>0$:
\[
H^0(\widetilde{Y}, \Sym^m \Omega_{\widetilde{Y}})
\subseteq H^0 (\widetilde{Y}, \Sym^m \Omega_{\widetilde{Y}} (log E))
= H^0(Y, \pi_* \Sym^m \Omega^1_{\widetilde{Y}}(log E))=
\]
\[
H^0(Y \setminus (Y^{sing} \cup F), \pi_* \Sym^m \Omega^1_{\widetilde{Y}}(log E))
\subseteq H^0(A, \Sym^m \Omega_A)^G \subseteq H^0(A, \Sym^m \Omega_A)
\]
(the inclusion 
$H^0(Y \setminus (Y^{sing} \cup F), \pi_* \Sym^m \Omega^1_{\widetilde{Y}}(log E))
\subseteq H^0(A, \Sym^m \Omega_A)^G$
holding by pulling back symmetric forms on $Y \setminus (Y^{sing} \cup F)$ 
to $G$-invariant symmetric forms on 
$A \setminus \gamma^{-1}(Y^{sing} \cup F)$ and extending them  to $A$).

Now, suppose by contradiction that $X$ has asymptotically generically generated cotangent bundle 
$\Omega_X$,
then the same holds for the cotangent bundle $\Omega_{\widetilde{Y}}$ as  
$\widetilde{Y}$ is smooth and birational to $X$, so for some $k>0$:
\[
\rk (\Sym^k \Omega_{\widetilde{Y}}) \leqslant \dim H^0(\widetilde{Y}, \Sym^k \Omega_{\widetilde{Y}})
\leqslant \dim H^0(A, \Sym^k \Omega_A)
= \rk ( \Sym^k \Omega_A),
\]
but the two vector bundles have the same rank, so all the inclusions above are equalities,
in particular 
$$H^0(A, \Sym^k \Omega_A)^G = H^0(A, \Sym^k \Omega_A)$$
and 
$$H^0(A, \Sym^k \Omega_A) \cong
H^0(\widetilde{Y}, \Sym^k \Omega_{\widetilde{Y}})
= H^0 (\widetilde{Y}, \Sym^k \Omega_{\widetilde{Y}} (log E))~.$$

Let us remark that by theorem \ref{smooth} above (and its proof), we know that if $Y = A/G$ is smooth
then $G$ acts by translations on $A$ and indeed the inclusions above are equalities in this case.

Let us show that if $Y = A/G$ is singular the   equalities above
 cannot hold.
 
First, notice that we can apply the same arguments as in the proof of
theorem \ref{smooth} and show that 
$H^0(A, \Sym^k \Omega_A)^G = H^0(A, \Sym^k \Omega_A)$ implies that 
$G$ acts on $H^0(A, \Omega_A)$ by homothethies, \emph{i.e.} 
the action $G \to GL(H^0(A, \Omega_A))$ is given by 
$g \mapsto \chi_g Id_{H^0(A, \Omega_A)}$ for some 
character $\chi \colon G \to \CC^*$.

Therefore we can study in detail the local structure of the singularities 
and their resolutions.
Let us observe that 
 the singular points of
$Y = A/G$ are isolated, 
and can be resolved by one blowing-up each singular point.

In fact if a point $a \in A$ has a non trivial stabilizer 
$\textrm{Stab}_a \subseteq G$,
the action 
$\textrm{Stab}_a  \to GL(H^0(A, \Omega_A))$ is faithful,
as we can assume that the action of $G$ on $A$ is faithful,
and a non-trivial element $g \in G$ that acts trivially on $GL(H^0(A, \Omega_A))$
would be a translation and could not be in $\textrm{Stab}_a$.
Then the stabilizer $\textrm{Stab}_a$ is a cyclic group,
as it injects in $ \mathbb{C}^* \cdot Id_{H^0(A, \Omega_A)}$.

So the stabilizer is a cyclic group that  acts on a sufficiently small 
coordinate  neighborhood of $a$ 
as mutiplication on the coordinates by roots of unity,
\[
g \colon (u_1 , \dots, u_n) \mapsto (\chi_g u_1 , \dots, \chi_g u_n)
\]
therefore the point $a$ is the only point in the neighborhood that is stabilized by an element of $G$,  
and the image of $a$ in $Y$ is an isolated cyclic quotient singularity of type 
$\frac{1}{r}(1,\dots,1)$, with $r$ being the order of the stabilizer.

This kind of singularities are are known to be  isomorphic to cones on a Veronese variety 
(cf. \cite{young})
and to be resolved by a single blowing up,
so we can study the differentials on the resolution of singularities
through the following commutative diagram:
\[
\squaremap{\widetilde{A}}{f}{\widetilde{Y}}{p}{\pi}{A}{\gamma}{Y}
\]
where 
$\pi$ is the resolution of $Y$ 
by blowing up the finite number of points in $Y^{sing}$,
$\gamma$ is the quotient map,
$\widetilde{A}$ is the blowing-up of $A$ in the points $\gamma^{-1} (Y^{sing})$,
and $f$ is a covering of degree $|G|$ ramified along the exceptional locus of $p$.
Furthermore the action of $G$ on $A$ extends to an action of $G$ on 
$\widetilde{A}$, whose quotient is 
$f \colon \widetilde{A} \to \widetilde{Y}$.

To construct the diagram and prove the constructions detailed above,
we can provide first a local description and then check that 
the action of $G$ extends globally to $\widetilde{A}$ and that its quotient is 
$\widetilde{Y}$.

Observe first that we can choose a sufficiently small
coordinate neighborhood 
$U_a \subset A$ around
a point $a\in A$ with nontrivial stabilizer $\textrm{Stab}_a \subseteq G$,
in such a way that $U_a$ is stable for the action of 
$\textrm{Stab}_a$, and that all of its translates by other elements of $G$ 
are pairwise disjoint. Then for an element $g \in G \setminus \textrm{Stab}_a$,
the other point $g\cdot a \in A$ has the conjugated of $\textrm{Stab}_a$ as stabilizer,
 the quotient of the nighborhood $U_a / \textrm{Stab}_a$ 
is isomorphic to the quotient 
$(g\cdot U_a) / g \cdot \textrm{Stab}_a \cdot g^{-1}$,
and these quotients are identified 
to neighborhoods of the point $\gamma (a) \in Y$.

Now an element in the stabilizer  acts on $U_a$ by multiplying all coordinates by the same root of the unity,
so we can lift the action of  $\textrm{Stab}_a$ on $U_a$
to the blowing-up $\widetilde{U_a}$ of the point $a$,
and remark that the exceptional divisor is fixed by this action,
so the quotient $\widetilde{U_a}/ \textrm{Stab}_a$ is isomorphic to the resolution 
of the singularity of the neighborhood $\gamma (U_a)$ obtained by 
blowing up the singular point.
If an element  $g$ is not in the stabilizer $\textrm{Stab}_a$, 
it will give an isomorphism 
between $\widetilde{U_a}$ and $\widetilde{g \cdot U_a}$.
Repeating this argument with all points having a non-trivial stabilzer and shrinking the open neighborhoods if needed,
we can see that the action of $G$ on $A$ extends to an action 
on $\widetilde{A}$.

The quotient $\widetilde{A} / G$ is normal, 
the map  $f \colon \widetilde{A} \to \widetilde{Y}$ is clearly $G$-invariant so it factors 
through $\widetilde{A} / G$, and the map $\widetilde{A} / G \to \widetilde{Y}$ is a 
bijective map between normal varieties therefore is an isomorphism.
So quotient map for the actopn of $G$ on $\widetilde{A}$ is $f \colon \widetilde{A} \to \widetilde{Y}$.

In order to describe the map $f$, let us observe that we can choose local coordinates 
$(u_1, \dots, u_{n}) \in A$ around a point $a \in A$ with non trivial stabilizer,
in such a way that 
the action of $\textrm{Stab}_a$ on $A$ is given by  $$g\cdot (u_1, \dots, u_{n}) = (\chi_g u_1, \dots, \chi_g u_n)~.$$
And we can choose local coordinates in the blowing-up  $(x_1, w_2 \dots, w_n) \in \widetilde{A}$ 
around a point $x \in \widetilde{A}$ in the exceptional divisor, with $p(x) = a \in A$, in such a way that 
the blowing-up map $p \colon \widetilde{A} \to A$ is given in local coordinates by

\[
(x_1, w_2 \dots, w_n) \mapsto (x_1, x_1 w_2 \dots, x_1 w_n) ~,
\]
\emph{i.e.}
 the exceptional divisor $Exc(p)$ is given locally  by the equation $x_1 =0$, and
 the blowing up is defined locally by
 $w_i = u_i / u_1$.

Then the action of $g \in \textrm{Stab}_a$ on $A$ lifts to 
an action on $\widetilde{A}$ which is given in local 
coordinates by:

\[
g \cdot (x_1, w_2 \dots, w_n) = (\chi_g x_1, w_2 \dots, w_n) ~,
\]
with fixed divisor $x_1=0$. Therefore we have a covering 
$f \colon \widetilde{A} \to \widetilde{Y}$ ramified along the divisor $x_1 =0$,
that in those local on $\widetilde{A}$ and local coordinates $(y_1, \dots , y_n)$ on $\widetilde{Y}$ is given by:

\[
(x_1, w_2 \dots, w_n) \mapsto (x_1^m, w_2 \dots, w_n) ~.
\]

Finally, the equalities above give us the following isomorphisms of vector spaces:

\[
\begin{array}{ccccc}
H^0(\widetilde{A}, \Sym^k \Omega_{\widetilde{A}}) & 
\stackrel{f^*}{{\longleftarrow}} &  
H^0(\widetilde{Y}, \Sym^k \Omega_{\widetilde{Y}})\\
\lmapup{p^*} & { } & \rmapdown{\cong}\\
H^0(A,\Sym^k \Omega_{A}) & \stackrel{\gamma^*}{{\longleftarrow}} &
 H^0(Y,\pi_* \Sym^k \Omega_{\widetilde{Y}}(log E))
\end{array}
\]

Now $p^*$ is an isomorphism as $p$ is birational,
however we see that 

\[
f^* \colon H^0(\widetilde{Y}, \Sym^k \Omega_{\widetilde{Y}}) 
\to H^0(\widetilde{A}, \Sym^k \Omega_{\widetilde{A}})
\]
 cannot be an isomorphism:
 in fact, given the local coordinates above, we can choose a base 
 $du_1, \dots, du_n$ of $H^0(A,\Omega_{A})$.
 Now the holomorphic symmetric differential $(du_1)^k$ is pulled back to 
 $(dx_1)^k \in H^0(\widetilde{A}, \Sym^k \Omega_{\widetilde{A}})$,
 and this cannot be the pull-back of a holomorphic symmetric differential 
 in $H^0(\widetilde{Y}, \Sym^k \Omega_{\widetilde{Y}})$,
in fact  $f^* dy_1 = mx_1^{m-1}dx_1$  vanishes along the divisor $x_1=0$,
and $f^* dy_j = d w_j$ for $j = 2, \dots, n$ 
so no holomorphic symmetric differential on $\widetilde{Y}$ can pull back to $(dx_1)^k$.

\end{proof}

This completes the proof of Theorem \ref{mainthnm}. 
Corollary \ref{maincor} follows as the Minimal Model Program 
and Abundance Conjecture hold in dimension 3, 
the remark on Iitaka dimension being given below.

\section{Remarks and examples}

\subsection{Iitaka dimension}
In \cite{mistrurbi} we defined the Iitaka dimension of an Asymptotically Generically Generated vector bundle $E$ on a variety $X$ 
as 
\[
k(X, E) :=\limsup_{k>0} (\dim \mathrm{Im} 
(\varphi_{\Sym^k E}))
\]
where 
$\varphi_{\Sym^k E}  \colon X \dashrightarrow \mathbb{G}r(H^0(X, \Sym^kE), 
\rk \Sym^k E) $ is the Kodaira map for the vector bundle $\Sym^k E$. 

We remarked that for strongly semiample vector bundles this is the same as 
the Iitaka dimension of the line bundle $\det E$, but we were not able to establish whether this holds in general for all (AGG) vector bundles.
However, this is the case if an AGG vector bundle has Iitaka dimension 0.

In fact we proved that for $k\gg 0$ the maps $\varphi_{\Sym^k E}$
factor through the Iitaka fibration 
$\varphi_{\det E, \infty}$ for the line bundle $\det E$, 
and a projection.
We were not able to prove that the projection they factor through is
a generically  finite map in general, 
but this must be the case if $k(X,E)=0$,
otherwise the image of the Iitaka fibration:
\[
\varphi_{\det E, \infty} = \varphi_{(\det E)^{\otimes m}} \colon X 
\dashrightarrow \pp (H^0(X, (\det E)^{\otimes m}))
\]
would be contained in a hyperplane of 
$\pp (H^0(X, (\det E)^{\otimes m}))$.

In particular a variety having AGG cotangent bundle of Iitaka dimension 0
must have Kodaira dimension 0.

\subsection{Compact K\"ahler case}

We remark that most of the techniques used hold in the projective case,
in particular Theorem \ref{gkpthm}, and this is the reason for 
stating the main results for smooth projective varieties.

However the Theorem proved in \cite{mistrurbi}, characterizing abelian varieties as those varieties of 
Kodaira dimension $0$ having some symmetric power of the cotangent bundle globally generated, 
works as well to characterize complex tori among compact K\"ahler varieties,
so
the main question of having a bimeromorphic characterization  makes sense also for K\"ahler 
or complex varieties.

According to \cite{hp} the MMP should work for K\"ahler varieties,
and does work in dimension 3, with a suitable characterization of 
torus quotients,
so in dimension 3 the same characterization 
for varieties bimeromorphic to complex tori should hold,
and it is worth asking whether a bimeromorphic (or biholomprphic) 
caracterization of complex tori holds in any dimension. 
We leave this to future investigations.

\subsection{Higher Kodaira dimension}

In \cite{hoering}, H\"oring gives a classification of varieties
$X$ of any Kodaira dimension $k(X)$
having (strongly) semiample cotangent bundle $\Omega_X$,
in particular he proves 
that these varieties are finite \'etale quotients of 
a product of an abelian variety
and a variety $Y$ with ample canonical bundle and same Kodaira dimension as $X$.

In case of Kodaira dimension 0 we proved in \cite{mistrurbi}
that we do not need to look at \'etale quotients,
and here we prove that we have a birational characterization as well.
It would be interesting to look at varieties with higher Kodaira dimension
and AGG cotangent bundle as well.

\subsection{Kummer varieties}

The Kummer surface,
being a $K3$-surface,
in known not to have any holomorphic symmetric differentials
(cf. \cite{kobayashi}). However we can use this as an example 
of theorem \ref{sing} and can be generalized to
the higher dimensional case:

\begin{example} 
\label{kummsurf}
	Let $A$ be an abelian surface, $Y = A/ \{\pm 1 \}$,
	and $\widetilde{X} \to Y$ the resolution of singularities, 
	where $\widetilde{X}$ is a K3 surface. 
	The singular surface $Y$ is not terminal, 
	however it does have klt singularities.
Then 
we can apply Theorem \ref{gkkthm} and have a diagram as in Theorem 	\ref{sing}:
\[
\squaremap{\widetilde{A}}{f}{\widetilde{X}}{p}{\pi}{A}{\gamma}{Y}
\]
with $f$ covering of degree 2, ramified over the exceptional 
locus of $p$.  For $m=2$ we have:

\[
H^0(\widetilde{X}, \Sym^2 \Omega_{\widetilde{X}} (log E)) 
= H^0(Y \setminus Y^{sing}, \Sym^2 \Omega_Y)
=
\]
\[
= H^0(A, \Sym^2 \Omega_A) ^{\pm 1} = H^0(A, \Sym^2 \Omega_A)
= H^0(\widetilde{A}, \Sym^2 \Omega_{\widetilde{A}}) 
\]

Now, following notations as in Theorem \ref{sing},
we see that a basis $du_1, du_2$ of $H^0(A, \Omega_A)$ 
pulls back to the basis
of 
$H^0(\widetilde{A}, \Omega_{\widetilde{A}})$
 that locally looks like
$dx_1,w_2 dx_1 + x_1dw_2$.

The basis $du_1^2, du_1du_2, du_2^2$ of
$H^0(A, \Sym^2 \Omega_A)$ is invariant for the action of
$\{ \pm 1 \}$, 
and  pulls back to a basis 
of $H^0(\widetilde{A}, \Sym^2  \Omega_{\widetilde{A}})$.
Now clearly $d x_1^2$ is not coming 
from a global symmetric differential 
$H^0 (\widetilde{X}, \Sym^2 \Omega_{\widetilde{X}})$
as $f^* d y_1 = 2x_1 dx_1$,
however $f^* (\frac{1}{y_1} dy_1^2) = 4 dx_1^2$, 
so we can see the reason why
\[
p^* H^0(A, \Sym^2 \Omega_A) \cong 
f^* H^0(\widetilde{X}, \Sym^2 \Omega_{\widetilde{X}} (log E))
\]

\begin{example}
\label{kummvar}
In the same way, given an abelian variety $A$ of dimension at least 3,
a Kummer variety $Y = A / \{ \pm 1 \}$ has isolated singularities,
and is terminal, so we can apply the constructions in Theorem \ref{sing}.
In that case the second symmetric power of the cotangent bundle cannot be 
generically generated, however we have that 
\[
p^* H^0(A, \Sym^2 \Omega_A) \cong 
f^* H^0(\widetilde{X}, \Sym^2 \Omega_{\widetilde{X}} (log E))
\]
as in the 2-dimensional case.

\end{example}

\begin{remark}

If we consider exterior powers instead of symmetric differentials,
then we do not need to look at the logarithmic complex:
in \cite{gkkp} it is proven that for any KLT variety $X$ with log resolution $\pi \colon \widetilde{X} \to X$ and for any 
$1 \leqslant p \leqslant \dim X$
the sheaf $\pi_* \Omega_{\widetilde{X}}^p $ is reflexive,
therefore:

\[
H^0(X \setminus X^{sing}, \Omega_X^p) \cong 
H^0(\widetilde{X}, \Omega_{\widetilde{X}}^p) ~.
\]

This is not the case for symmetric powers, 
as it is shown in Examples \ref{kummsurf} and \ref{kummvar} above:
given an abelian surface $A$ and the corresponding
Kummer surface $\widetilde{X} $, resolution of the quotient 
$Y = A / \{ \pm 1 \}$,
 we have:

\[
H^0(A, \Omega_A^2)  \cong  H^0(A, \Omega_A^2)^{\{ \pm 1 \} } 
\cong H^0(Y \setminus Y^{sing}, \Omega_Y^2) \cong \CC \cong
H^0({\widetilde{X}}, \Omega_{\widetilde{X}}^2)
 ~,
\]
however 
\[
\mathbb{C}^3 \cong H^0(Y \setminus Y^{sing}, \Sym^2 \Omega^1_{Y}) \cong  
H^0(A, \Sym^2 \Omega^1_A) \neq H^0(X, \Sym^2 \Omega^1_X) = 0  ~,
\]
and similarly in the higher dimensional case.
\end{remark}

\end{example}

%
%

\begin{thebibliography}{BKK{\etalchar{+}}15}

\bibitem[BKK{\etalchar{+}}15]{bkkmsu}
Thomas Bauer, S\'andor~J. Kov\'acs, Alex K\"uronya, Ernesto~C. Mistretta,
  Tomasz Szemberg, and Stefano Urbinati, \emph{On positivity and base loci of
  vector bundles}, Eur. J. Math. \textbf{1} (2015), no.~2, 229--249.

\bibitem[GKK10]{GKK}
Daniel Greb, Stefan Kebekus, and S\'andor~J. Kov\'acs, \emph{Extension theorems
  for differential forms and {B}ogomolov-{S}ommese vanishing on log canonical
  varieties}, Compos. Math. \textbf{146} (2010), no.~1, 193--219.

\bibitem[GKKP11]{gkkp}
Daniel Greb, Stefan Kebekus, S\'andor~J. Kov\'acs, and Thomas Peternell,
  \emph{Differential forms on log canonical spaces}, Publ. Math. Inst. Hautes
  \'Etudes Sci. (2011), no.~114, 87--169.

\bibitem[GKP16]{GKP}
Daniel Greb, Stefan Kebekus, and Thomas Peternell, \emph{\'etale fundamental
  groups of {K}awamata log terminal spaces, flat sheaves, and quotients of
  abelian varieties}, Duke Math. J. \textbf{165} (2016), no.~10, 1965--2004.

\bibitem[HP17]{hp}
Andreas H{\"{o}}ering and Thomas Peternell, \emph{Bimeromorphic geometry of
  k{\"{a}}hler threefolds}, e-print: \texttt{arXiv:1701.01653}, January 2017.

\bibitem[H{\"{o}}r13]{hoering}
Andreas H{\"{o}}ring, \emph{Manifolds with nef cotangent bundle}, Asian J.
  Math. \textbf{17} (2013), no.~3, 561--568.

\bibitem[Kaw85]{kawamata}
Yujiro Kawamata, \emph{Minimal models and the {K}odaira dimension of algebraic
  fiber spaces}, J. Reine Angew. Math. \textbf{363} (1985), 1--46.

\bibitem[Kob80]{kobayashi}
Shoshichi Kobayashi, \emph{The first {C}hern class and holomorphic symmetric
  tensor fields}, J. Math. Soc. Japan \textbf{32} (1980), no.~2, 325--329.

\bibitem[MU17]{mistrurbi}
Ernesto~C. Mistretta and Stefano Urbinati, \emph{Iitaka fibrations for vector
  bundles}, International Mathematics Research Notices (2017), rnx239.

\bibitem[Rei87]{young}
Miles Reid, \emph{Young person's guide to canonical singularities}, Algebraic
  geometry, {B}owdoin, 1985 ({B}runswick, {M}aine, 1985), Proc. Sympos. Pure
  Math., vol.~46, Amer. Math. Soc., Providence, RI, 1987, pp.~345--414.

\end{thebibliography}
%

\newcommand{\etalchar}[1]{$^{#1}$}
\providecommand{\bysame}{\leavevmode\hbox to3em{\hrulefill}\thinspace}
\providecommand{\MR}{\relax\ifhmode\unskip\space\fi MR }
\providecommand{\MRhref}[2]{%
  \href{http://www.ams.org/mathscinet-getitem?mr=#1}{#2}
}
\providecommand{\href}[2]{#2}

\end{document}